\documentclass[11pt]{article}
\usepackage{latexsym,amsmath,amssymb,amsthm}
\usepackage{setspace}
\usepackage{amssymb, amsmath, amsthm, graphicx,mathrsfs}
\usepackage[left=1in,top=1in,right=1in,bottom=1in]{geometry}
\usepackage{subfigure}
\usepackage{titlesec}
\titleformat{\subsection}[hang]
  {\normalfont\bfseries}{\thesubsection}{1em}{}
  \usepackage{enumitem}
\usepackage{xcolor}

      \newtheorem{theorem}{Theorem}
      \newtheorem{lemma}[theorem]{Lemma}

      \newtheorem{claim}[theorem]{Claim}

\def\ex{{\rm{ex}}}
\def\ttt{{s}}
\def\SSS{{t}}

\def\BB{{\mathcal B}}
\def\EE{{\mathcal E}}
\def\FF{{\mathcal F}}
\def\GG{{\mathcal G}}
\def\Hh{{\mathcal H}}
\def\HH{\widehat{\mathcal H}}
\def\JJ{{\mathcal J}}
\def\PP{{\mathcal P}}

\def\TT{{\mathcal T}}
\def\UU{{\mathcal U}}

\begin{document}

\pagestyle{myheadings}
\markright{{\small{{\sc F\"uredi and Kostochka:   Bushes}. \enskip December 6, 2023}}}

\title{\vspace{-0.5in} Tur\' an number for bushes}

\author{
{\large{Zolt\'an F\"uredi}}\thanks{
\footnotesize {Alfr\'ed R\'enyi Institute of Mathematics, Budapest, Hungary.
E-mail:  \texttt{z-furedi@illinois.edu}.
Research partially supported by National Research, Development and Innovation Office NKFIH grants  132696 and 133819.}}
\and
{\large{Alexandr Kostochka}}\thanks{
\footnotesize {University of Illinois at Urbana-Champaign, Urbana, IL 61801
 and Sobolev Institute of Mathematics, Novosibirsk 630090, Russia. E-mail: \texttt {kostochk@illinois.edu}.
Research supported in part by NSF
grant DMS-2153507 and NSF RTG grant DMS-1937241.
}}
}

\date{}    
\maketitle
\vspace{-0.3in}

\begin{abstract}
Let $ a,b \in {\bf Z}^+$,
 $r=a + b$, and let $T$ be a tree
  with parts $U = \{u_1,u_2,\dots,u_s\}$
and $V = \{v_1,v_2,\dots,v_t\}$.
Let $U_1, \dots ,U_s$ and $V_1, \dots, V_t$ be  disjoint sets, such that  {$|U_i|=a$  and $|V_j|=b$ for all $i,j$}.
The  {\em  $(a,b)$-blowup} of $T$  is the
$r$-uniform hypergraph  with edge set
$ {\{U_i \cup V_j :
u_iv_j \in E(T)\}.}$

We use the $\Delta$-systems method to prove the following Tur\' an-type result.
Suppose  $a,b,\ttt \in {\bf Z}^+$,  $r=a+b\geq 3$,{ $a\geq 2$,} and
 $T$ is a fixed tree of diameter $4$ in which the degree of the center vertex is $\ttt $.
Then there exists a $C=C(r,\ttt ,T)>0$ such that
$ |\mathcal{H}|\leq  (\ttt -1){n\choose r-1} +Cn^{r-2}$ for every
 $n$-vertex $r$-uniform hypergraph  $\mathcal{H}$  {not containing
 an $(a,b)$-blowup of $T$}. This is {asymptotically exact} when $\ttt \leq |V(T)|/2$.
  A stability result is also presented.

\end{abstract}

\section{Introduction}
\subsection{Basic definitions and notation}

 An $r$-uniform hypergraph (an {\em $r$-graph}, for short), is a family of $r$-element subsets of a finite set.
We associate an $r$-graph $\Hh$ with its edge set and call its vertex set $V(\Hh)$.
Often we take $V(\Hh)=[n]$, where  $[n]:=\{ 1, 2, 3,\dots, n\}$.
Given an $r$-graph $\FF$,
let {\em the Tur\' an number} of $\FF$, $\ex_r(n,\FF)$, denote the maximum number of edges in an $r$-graph on $n$ vertices that does not contain a copy of $\FF$.

Since
 a (graph) tree is connected and bipartite, it uniquely defines the parts in its bipartition. So, we say a tree
 $T$ is an {\em $(s,t)$-tree} if one part of $V(T)$ has $s$ vertices and the other has $t$ vertices.

Let $s,t, a,b > 0$ be integers,
 $r=a + b$, and let $T = T(U,V)$ be an $(s,t)$-tree
  with parts $U = \{u_1,u_2,\dots,u_s\}$
and $V = \{v_1,v_2,\dots,v_t\}$.
Let $U_1, \dots ,U_s$ and $V_1, \dots, V_t$ be pairwise disjoint sets, such that $|U_i|=a$  and $|V_j|=b$ for all $i,j$.
So $\left| \bigcup U_i\cup V_j \right| = as +bt $.
The {\em  $(a,b)$-blowup} of $T$, denoted by $\TT(T,a,b)$, is the
$r$-uniform hypergraph  with edge set
\[ \TT(T,a,b):= \{U_i \cup V_j :
u_iv_j \in E(T)\}.\]

\begin{figure}
\begin{center}
    \includegraphics[height=1in]{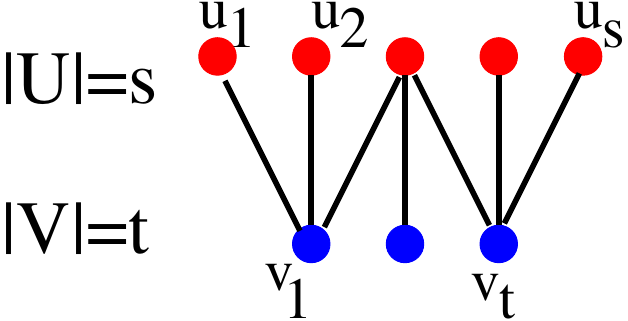} ${\bf \quad \Rightarrow\quad }$  \includegraphics[height=1.05in]{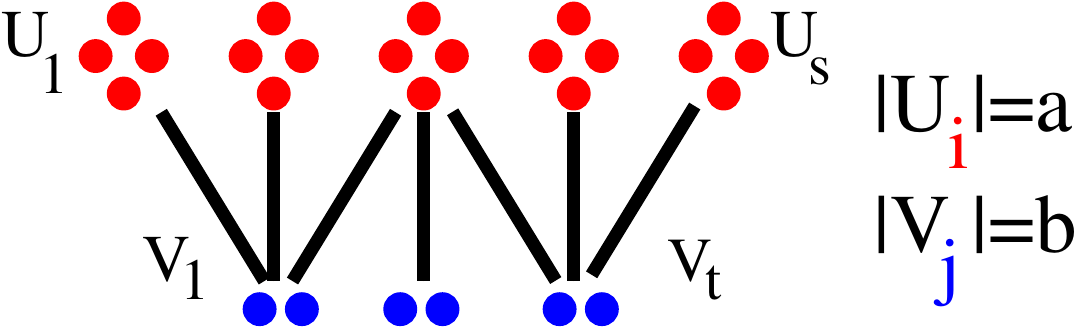}
\caption{An example of a $(4,2)$-blowup. \label{fig1}}
\end{center}
\end{figure}

The goal of this paper is to find the asymptotics of the Tur\' an number for $(a,b)$-blowups of many trees of radius $2$ using the $\Delta$-systems method.
Earlier, $(a,b)$-blowups of different classes of  trees and different pairs $(a,b)$ were considered in~\cite{FJKMV5}. The main result in~\cite{FJKMV5} is the following.

\begin{theorem}[\cite{FJKMV5}]\label{th:main1}
Suppose  $r \geq 3$, $s,t \geq 2$, $a + b = r$, $b < a < r$.
Let $T$ be an $(s,t)$-tree  and let ${\mathcal T}=\TT(T,a,b)$ be its $(a,b)$-blowup.
Then (as $n\to \infty$) any
$\mathcal{T}$-free $n$-vertex $r$-graph $\Hh$ satisfies
\[ |\Hh| \leq (t - 1){n \choose r-1} + o(n^{r-1}).\]
This is asymptotically sharp whenever $t\leq s$.
\end{theorem}

This theorem asymptotically settles about a half of possible cases, but when $t>s$
it is expected that the asymptotic is different.
More is known on $(a,b)$-blowups of paths.

Let $P_\ell$ denote the (graph) path with $\ell$ edges.
The first edge of the path corresponds to $A_1\cup B_1$, the second edge to $B_1\cup A_2$, etc.
The case of $P_2$ was resolved asymptotically by Frankl~\cite{Frankl1977} (for $b=1$)  and by Frankl and F\" uredi~\cite{FF85} (for all $1\leq a\leq r-2$ and $b=r-a$):
\begin{equation*} 
   \ex_r(n, \TT(P_2,a,b)) = \Theta\left( n^{\max\{ a-1,b \}}\right).
   \end{equation*}
The case of $P_3$ was fully solved for large $n$ by  F\"uredi and \"Ozkahya~\cite{FurOzk}.
They showed that for fixed $1\leq a,b <r$ with $r=a+b\geq 3$ and for $n> n_0(r)$,
\begin{equation*}
   \ex_r(n,\TT( P_3,a,b)) = \binom{n-1}{r-1}.
   \end{equation*}
For longer paths, the following was proved in~\cite{FJKMV5}.

\begin{theorem}[Theorem 1 in \cite{FJKMV5}]\label{th:path1}
Let  $a+b=r$, $a,b\geq 1$ and $\ell \geq 3$. Suppose further that
(i) $\ell$ is odd, or
	(ii)  $\ell$ is even and $a>b$, or (iii)  $(\ell, a,b) = (4,1,2)$.
	Then
$$\ex_r(n,\TT(P_{\ell},a,b)) = { \left \lfloor\frac{\ell-1}{2}\right \rfloor}{n \choose r - 1}+ o(n^{r - 1}).$$
\end{theorem}

So, the situation with blowups of $P_\ell$ is not resolved for the case when $\ell\geq 4$ is even and $a\leq b$ apart from
the case $(\ell, a,b) = (4,1,2)$.

\bigskip
In this paper, we 
consider $(a,b)$-blowups of 
trees of radius $2$.

 A {\em graph bush} $B_{\ttt ,h}$ is the radius 2 tree obtained from the star $K_{1,\ttt }$ by joining each  vertex of degree one to
 $h$ new vertices. So $B_{\ttt ,h}$ has $1+\ttt  +\ttt h$ vertices. Let $\SSS=1+\ttt h$. Then $B_{\ttt ,h}$ is an $(s,t)$-tree with $\SSS >\ttt $.

Suppose that $a,b,\ttt,h$ are positive integers, $a+b=r$ and $\ttt \geq 2$.
By   $(a,b,\ttt,h)$-{\em bush}, $\BB_{\ttt  ,h}(a,b)$, we will call the $(a,b)$-blowup of $B_{\ttt  ,h}$.
This means the center vertex of  $B_{\ttt  ,h}$ is replaced by an $a$-set $A$,
its neighbors by the $b$-sets $B_1, \dots, B_\ttt $ and its second
neighbors by  $a$-sets $A_{i,j}$, $i\in [\ttt]$, $j\in [h]$.
In particular, the $(a,b)$-blowup of  the {path}  $P_4$  is the $(a,b,2,1)$-{bush} $\BB_{2,1}(a,b)$.

\subsection{New results, bushes and shadows}

Since $\BB_{\ttt  ,h}(a,b)$ has $\ttt $ disjoint edges $B_{i}\cup A_{i,1}$ for $i=1,\ldots,\ttt $, the example of the $r$-uniform hypergraph with vertex set $[n]$ in which every edge intersects the set $[\ttt -1]$ shows that
\begin{equation}\label{eq:11}
\ex_r(n,\BB_{\ttt  ,h}(a,b))\geq {n\choose r}-{n-\ttt +1\choose r}\sim (\ttt -  1){n\choose r-1}.
  \end{equation}

We will use the $\Delta$-systems approach to show that this is asymptotically correct in many cases.
For $a>b\geq 2$  the asymptotic equality follows from Theorem~\ref{th:main1}.
In this paper we deal with {\em all} cases and also present a somewhat refined result by
 considering  shadows of hypergraphs.

For an $r$-graph $\mathcal{H}$ the {\em shadow}, $\partial \mathcal{H}$, is the collection of $(r-1)$-sets that lie in some edge of $\mathcal{H}$.
 Our first main result is the following.

\begin{theorem}\label{th:main}
Suppose that $a,b,\ttt,h$ are positive integers,  $r=a+b\geq 3$.
Also suppose that in case of $(a,b)= (1,r-1)$ we have $h=1$.
Then there exists a $C=C(r,\ttt,h)>0$ such that every $n$-vertex $r$-uniform family $\mathcal{H}$
satisfying
\begin{equation}\label{eq:12}
    |\mathcal{H}|> (\ttt -  1)|\partial \mathcal{H}| +Cn^{r-2}
\end{equation}
 contains
 the bush  $\BB_{\ttt  ,h}(a,b)$.
 \end{theorem}

This implies that~\eqref{eq:11} is asymptotically exact in these cases as $r,\ttt,h$ are fixed and $n\to \infty$.
Note that for $(a,b)=(1,r-1)$ our proof works {only for $h=1$.}
In fact, in this case the theorem  {does not hold for $h\geq2$. }
An example is this:  $V(\Hh)=[n], A=[\ttt -1]$ and  $\Hh =\EE_1\cup\EE_2$, where
$\EE_1$ is the set of $r$-subsets of $[n]$ with  {exactly one vertex in $A$} and
$\EE_2$ is the Steiner system $S_1(n-\ttt +1,r,r-1)$ on $[n]\setminus A$.
This example has asymptotically  {$(\ttt -  1+\frac{1}{r}){n\choose r-1}$ edges} and does not contain
$\BB_{\ttt  ,2}(1,r-1)$.
By increasing $h$, we can get examples  {without $\BB_{\ttt  ,h}(1,r-1)$} that have even more edges.

McLennan~\cite{Tree} proved that in the graph case ($a=b=1$), 
$\ex_r(n, B_{\ttt  ,h})= \frac{1}{2}(\ttt +\ttt h-1)n+O(1)$. 
Vertex-disjoint unions of complete graphs $K_{\ttt  +\ttt h}$ are extremal.
For the case $\ttt =1$ the restriction~\eqref{eq:12} is too strong, $\ex_r(n, \BB_{1,h}(a,b))= O(n^{r-2})$ is known for $a\geq 2$. Even  better bounds were proved in~\cite{FF60}. So we
suppose that $\ttt \geq 2$, $r\geq 3$.

Since each tree of diameter $4$ with the degree of the center equal to $\ttt $ is a subgraph of a graph bush $B_{\ttt  ,h}$ for some $h$, Theorem~\ref{th:main}
 yields the following  more
 general result.

\begin{theorem}\label{th:main2}
Suppose  $a,b,\ttt\in {\bf Z}^+$,  $r=a+b\geq 3$ and { $a\geq 2$.}
Let $T$ be a fixed tree of diameter $4$ in which the degree of the center vertex is $\ttt $.
Then there exists a $C=C(r,\ttt,T)>0$ such that every $n$-vertex $r$-graph  $\mathcal{H}$
satisfying
\begin{equation*}  
    |\mathcal{H}|>  {(\ttt -  1){n\choose r-1} +Cn^{r-2}}
\end{equation*}
  contains
 an $(a,b)$-blowup of $T$.
The coefficient $(\ttt-1)$ is the best possible (as $n\to \infty$).
\end{theorem}

We also use the $\Delta$-systems approach to show that for $a,b\geq 2$, $a+b=r$, an $r$-uniform hypergraphs without $\BB_{\ttt  ,h}(a,b)$ of cardinality "close" to extremal contains vertices of "large" degrees.

\begin{theorem}\label{th:stab}
Suppose that $a,b,\ttt,h$ are positive integers, $a,b\geq 2$ and  $r=a+b\geq 5$.
Then for any $C_0>0$
there exist $n_0>0$ and $C_1>0$  such that the following holds.
If  $n>n_0$, $\mathcal{H}$ is an $n$-vertex $r$-uniform family not containing
 a bush  $\BB_{\ttt  ,h}(a,b)$ and $ |\mathcal{H}|> (\ttt -  1){n\choose {r-1}}-C_0 n^{r-2}$,  then
there are $\ttt -1$ vertices in $[n]$ each of which is contained in at least ${n\choose {r-1}}-C_1 n^{r-2}$ edges of $\Hh$.
\end{theorem}

The structure of this paper is as follows. In the next section, we discuss the $\Delta$-system method and present a lemma
by F\" uredi~\cite{Furedi1} from 1983 that will be our main tool. In Section~\ref{lemmas} we describe properties of so called intersection structures. It allows us to prove the main case of Theorem~\ref{th:main} (the case $a\geq 2$) in
Section~\ref{maint} and the case of $a=1$ and $h=1$ in
Section~\ref{s5}. In
Section~\ref{stabi} we prove Theorem~\ref{th:stab}.

\section{Definitions for the $\Delta$-system method and a lemma}

A family of sets $\{F_1,\ldots,F_q\}$ is a {\em $q$-star} or a {\em $\Delta$-system} or a {\em $q$-sunflower 
 with kernel $A$}, if
$F_i\cap F_j=A$ for all $1\leq i<j\leq q$. The sets $F_i\setminus A$ are called {\em petals}.

For a member $F$ of a family $\mathcal{F}$, let the {\em intersection structure of $F$ relative to $\mathcal{F}$} be
$$ \mathcal{I}(F,\mathcal{F})=\{F\cap F': F'\in \mathcal{F}\setminus \{F\}\}.$$

An $r$-uniform family $\mathcal{F}\subseteq {[n]\choose r}$ is $r$-{\em partite} if there exists a partition
$(X_1,\ldots,X_r)$     of the vertex set $[n]$ such that $|F\cap X_i|=1$ for each
$F\in \mathcal{F}$ and each $i\in [r]$.
For a partition $(X_1,\ldots,X_r)$ of $[n]$ and a set $S\subseteq [n]$, the {\em pattern} $\Pi(S)$ is the set
$\{i\in [r]: S\cap X_i\neq \emptyset\}$. Naturally, for a family $\mathcal{L}$ of subsets of $[n]$,
$$\Pi(\mathcal{L})=\{\Pi(S):S\in \mathcal{L}\}\subseteq 2^{[r]}.$$

\begin{lemma}[The intersection semilattice lemma (F\" uredi~\cite{Furedi1})] \label{Zoll}
For any positive integers $q$ and $r$, there exists a positive constant $c(r,q)$ such that every
family $\mathcal{F}\subseteq {[n]\choose r}$ contains a subfamily $\mathcal{F}^*\subseteq \mathcal{F}$ satisfying

1.  $|\mathcal{F}^*|\geq c(r,q)|\mathcal{F}|$.

2. $\mathcal{F}^*$  is $r$-partite, together with an $r$-partition $(X_1,\ldots,X_r)$.

3. There exists a family $\mathcal{J}$ of proper subsets of  $[r]$ such that  $\Pi(\mathcal{I}(F,\mathcal{F}^*))=\mathcal{J}$
holds for all $F\in \mathcal{F}^*$.

4. $\mathcal{J}$ is closed under intersection, i.e., for all $A,B\in \mathcal{J}$ we have $A\cap B\in  \mathcal{J}$, as
well.

5. For any $F\in \mathcal{F}^*$ and each $A\in \mathcal{I}(F,\mathcal{F}^*)$, there is a $q$-star in $\mathcal{F}^*$
containing $F$ with kernel $A$.
\end{lemma}

{\bf Remark 1.} {\em The proof of Lemma~\ref{Zoll} in~\cite{Furedi1} yields that if $\mathcal{F}$ itself is $r$-partite with an $r$-partition $(X_1,\ldots,X_r)$, then the $r$-partition in the statement can be taken the same.}

\medskip
{\bf Remark 2.} {\em By definition,  if for some $M\subset [r]$ none of the  members of the family $\mathcal{J}$ of proper subsets of  $[r]$ in
Lemma~\ref{Zoll} contains $M$, then for any two sets $F_1,F_2\in \mathcal{F}^*$, their intersections with
$\bigcup_{j\in M}X_j$ are distinct. It follows that if $|M|=m$, then
 $|\mathcal{F}^*|\leq \prod_{j\in M}|X_j| \leq \left(\frac{ n-(r-m)}{m}\right)^m$.
Thus, if $|\mathcal{F}^*|> \left(\frac{ n-r+m}{m}\right)^m$, then every $m$-element subset of $[r]$ is contained in some
$B\in \mathcal{J}$.}

Call a family $ \mathcal{J}$ of proper subsets of $[r]$ {\em $m$-covering} if every $m$-element subset of $[r]$ is contained in some
$B\in \mathcal{J}$. In these terms, Remark 2 says that
\begin{equation}\label{cover}
\mbox{\em if $|\mathcal{F}^*|> \left(\frac{ n-r+m}{m}\right)^m$, then the corresponding $ \mathcal{J}$ is $m$-covering.}
\end{equation}

\medskip
For $k=0,1,\ldots,r$, define the family $ \mathcal{J}^{(k)}$ of proper subsets of $[r]$ as follows. It contains  \newline
 (a) the sets $[r]\setminus \{i\}$ for $1\leq i\leq k$,  \newline
 (b) all $(r-2)$-element subsets of $[r]$ containing $\{1,2,\ldots,k\}$, and  \newline
 (c) all the intersections of these subsets.

By definition, each $ \mathcal{J}^{(k)}$ is  $(r-2)$-covering. Moreover,
\begin{equation}\label{cover2}
\parbox{14.5cm}{\em each $(r-2)$-covering  family  of proper subsets of $[r]$ closed under intersections contains a subfamily isomorphic to some $ \mathcal{J}^{(k)}$.}
\end{equation}
Indeed, if an $(r-2)$-covering  family $ \mathcal{J}$ of proper subsets of $[r]$ contains exactly $k$ sets of size $r-1$, 
then it must contain as members all  $(r-2)$-element subsets of $[r]$ not contained in these $k$ sets, so properties (a) and (b) of the definition hold. Part (c) follows since $ \mathcal{J}$ is closed under intersections.

\section{General claims on intersection structures.}\label{lemmas}

Call a set  $B$ a  {\em $(b,q)$-kernel} in a set system $\mathcal{F}$ if $B$ is the kernel of size $b$ in a sunflower with $q$ petals formed by members of $\mathcal{F}$.

\begin{lemma}\label{no+b} If $a+b=r$ and an $r$-uniform family $\mathcal{H}$ does not contain $\BB_{\ttt  ,h}(a,b)$, then there do not exist
disjoint sets $A_0,B_1,B_2,\ldots,B_\ttt $ with $|A_0|=a$, $|B_1|=\ldots =|B_\ttt |=b$ such that all $B_1,\ldots,B_\ttt $ are $(b,\ttt hr)$-kernels in $\mathcal{H}$ and the sets $A_0\cup B_1,\ldots ,A_0\cup B_\ttt $ are edges of $\mathcal{H}$.
\end{lemma}

\begin{proof} Suppose, there are such disjoint sets $A_0,B_1,B_2,\ldots,B_\ttt $. Let $D_0=A_0\cup \bigcup_{j=1}^\ttt B_j$. For $i=1,\ldots,\ttt $, do the following. Since $B_i$ is a $(b,\ttt hr)$-kernel and $|D_{i-1}\setminus B_i|=a+(\ttt -  1)b +(i-1)ha\leq \ttt  h r -h$, there exist $h$ petals $A_{i,j}$ ($1\leq j\leq h$) of a $\ttt hr$-sunflower with kernel $B_i$ that are disjoint from $D_{i-1}$. Let $D_i=D_{i-1}\cup\bigcup_{j=1}^h A_{i,j}$.
After $\ttt $ steps, we find a $\BB_{\ttt  ,h}(a,b)$ whose  edges are $A_0\cup B_i$ and $B_i\cup A_{i,j}$ for $i=1,\ldots,\ttt $, $j=1,2,\ldots,h$.
\end{proof}

Suppose $a+b=r$ and $ \mathcal{G}\subset {[n]\choose r}$ with $|\mathcal{G}|>\frac{1}{c(r,\ttt hr)}n^{r-2}$
does not contain $\BB_{\ttt  ,h}(a,b)$.
By Lemma~\ref{Zoll} and~\eqref{cover}, there is $\mathcal{G}^*\subseteq \mathcal{G}$
satisfying the lemma such that the corresponding family $\mathcal{J}$ of proper subsets of  $[r]$ is $(r-2)$-covering.
Let $(X_1,\ldots,X_r)$ be the corresponding partition.

\begin{lemma}\label{noa+b} Family $\mathcal{J}$ does not contain disjoint members $A$ and $B$ such that $|A|=a$ and $|B|=b$.
\end{lemma}

\begin{proof} Suppose, it does. By renaming the elements of $\mathcal{J}$, we may assume that $A=\{1,\ldots,a\}$ and
$B=\{a+1,\ldots,r\}$. Let $X=\{x_1,\ldots,x_r\}\in\mathcal{G}^*$, where $x_i\in X_i$ for all $i$.
 Since $A\in \mathcal{J}$, $\{x_1,\ldots,x_a\}$ is an $(a,\ttt hr)$-kernel in $\mathcal{G}^*$. Let $B_1,\ldots,B_\ttt $ be some $\ttt $ petals in the sunflower with kernel $\{x_1,\ldots,x_a\}$. As $[r]\setminus [a]=B\in \mathcal{J}$, each of $B_1,\ldots,B_\ttt $
 is a $(b,\ttt hr)$-kernel in $\mathcal{G}^*$, contradicting Lemma~\ref{no+b}.
\end{proof}

\begin{lemma}\label{inter} If $a+b=r$ and $2\leq a,b\leq r-2$, then for each $0\leq k\leq r-2$ and for $k=r$ the family
$ \mathcal{J}^{(k)}$ has
 disjoint members $A$ and $B$ such that $|A|=a$ and $|B|=b$, unless $(r,a,b,k)=(4,2,2,1)$.

  If $(a,b)=(1,r-1)$ or $(a,b)=(r-1,1)$ and $r\geq 3$, then for each $1\leq k\leq r-2$ and for $k=r$ the family
$ \mathcal{J}^{(k)}$ has
 disjoint members $A$ and $B$ such that $|A|=a$ and $|B|=b$.
\end{lemma}

\begin{proof} If $k\geq a$, then we let $A=[a]$, $B=[r]\setminus [a]$, and represent them as follows:
$$A=\bigcap_{k+1\leq i<i'\leq r}([r]\setminus \{i,i'\})\cap \bigcap_{a+1\leq i\leq k}([r]\setminus \{i\}),\quad
B=\bigcap_{1\leq i\leq a}([r]\setminus \{i\}).$$
If $k\geq b$, then we have a symmetric representation.
 In particular, this proves the claim for $(a,b)=(1,r-1)$ and $(a,b)=(r-1,1)$.

 If $k\leq a-2$, then we again let $A=[a]$, $B=[r]\setminus [a]$, but represent them as follows (using that $a\leq r-2$ and $k\leq a-2$):
$$A=\bigcap_{a+1\leq i<i'\leq r}([r]\setminus \{i,i'\}),\quad
B=\bigcap_{1\leq i\leq k}([r]\setminus \{i\})\cap \bigcap_{k+1\leq i<i'\leq a}([r]\setminus \{i,i'\}).$$
By symmetry, the only remaining case is that $k=a-1=b-1$. So $r$ is even, and $a=b=r/2=k+1$. In this case, if $r>4$, then
 we let $A=[k-1]\cup \{k+1,k+2\}$, $B=([r]\setminus [k+2])\cup \{k\}$, and represent them as follows:
$$A=\bigcap_{k+3\leq i<i'\leq r}([r]\setminus \{i,i'\})\cap ([r]\setminus \{k\}),\quad
B=\bigcap_{1\leq i\leq k-1}([r]\setminus \{i\})\cap ([r]\setminus \{k+1,k+2\}).$$
\end{proof}

\section{Proof of the main Theorem for $a>1$} \label{maint} 

In this section we prove the main part of Theorem~\ref{th:main}: the case of $2\leq a\leq r-1$. In Subsection~\ref{41} we describe a procedure of partitioning $\Hh$ into structured subfamilies, and in the next three subsections we  use this partition to find the required bush or to get a contradiction to~\eqref{eq:12}.
We distinguish three cases: (i) $a,b\geq 2$ and $r\geq 5$ (discussed in Subsection~\ref{ss41}), (ii) $(a,b)=(2,2)$ and $r=4$ (Subsection~\ref{ss42}), and  (iii) $(a,b)=(r-1,1)$ and $r\geq 3$ (Subsection~\ref{ss43}).

\subsection{Basic procedure}\label{41}

Assume that $2\leq a\leq r-1$ and that an $n$-vertex $r$-uniform family $\mathcal{H}$ 
satisfies~\eqref{eq:12} but  does not contain $\BB_{\ttt  ,h}(a,b)$.
Define $C=C(r,\ttt,h) := 1/c(r,\ttt hr)$, where $c$ is from Lemma~\ref{Zoll}.
For any $r$-uniform family  $\mathcal{G}$, let $\mathcal{G}^*$ denote a family
satisfying Lemma~\ref{Zoll} and  $\mathcal{J}(\mathcal{G}^*)\subset 2^{[r]}$ denote the corresponding intersection structure.

Do the following procedure. Let $\mathcal{H}_1=\mathcal{H}^*$
and $\JJ_1=\mathcal{J}(\mathcal{H}^*)$. 
For $i=1,2,\ldots,$ if 
    $|\mathcal{H} \setminus \bigcup_{j=1}^i \mathcal{H}_j|\leq C n^{r-2}$,
then stop and let $m:=i$ and 
    $\mathcal{H}_0=\mathcal{H}\setminus \bigcup_{j=1}^i \mathcal{H}_j$;
otherwise, let 
    $\mathcal{H}_{i+1}:=(\mathcal{H}\setminus \bigcup_{j=1}^i \mathcal{H}_j)^*$
    and $\JJ_{i+1}=\mathcal{J}((\mathcal{H}\setminus \bigcup_{j=1}^i \mathcal{H}_j)^*)$.

This procedure provides a partition of $\mathcal{H}$, 
    $\mathcal{H}=\bigcup_{i=0}^m \mathcal{H}_i$.
Let $\HH$ denote $\bigcup_{i=1}^m \mathcal{H}_i$.
By definition, 
    $|\mathcal{H}_0| \leq Cn^{r-2}$,
so we get
\begin{equation}\label{eq:41}
|\HH| +  Cn^{r-2} \geq  | \mathcal{H}|.
  \end{equation}

\subsection{Case of  $a,b\geq 2$ and $r\geq 5$}\label{ss41}
Here $2\leq a,b\leq r-2$ and $(a,b)\neq (2,2)$, so Lemmas~\ref{noa+b} and~\ref{inter} imply that for each $1\leq i\leq m$, $\mathcal{J}(\mathcal{H}_i)$ has exactly $r-1$ $(r-1)$-subsets, it contains $ \mathcal{J}^{(r-1)}$.
Hence for each hyperedge $E\in \mathcal{H}_i\subset \HH $ ($1\leq i\leq m$) there exists an element $c(E)\in E$ such that
 each proper subset of $E$ containing $c(E)$ is a kernel of an $\ttt hr$-star in  $\mathcal{H}_i$.
Beware of the fact that although each $\mathcal{H}_i$ is $r$-partite, the partitions might differ for different values of $i$.
This does not cause any problem in our argument, we only need the existence of the element $c(E)$.

Define the function $\alpha$ on ${[n] \choose r-1}$ as follows:
For each $(r-1)$-set $Y\subset [n]$, let $\alpha(Y)$ be the number of edges $E\in \HH$ with $Y=E\setminus \{ c(E)\}$.

\begin{claim}\label{distinct}
$\alpha(Y)\leq \ttt -1$, for each $(r-1)$-subset $Y$ of $[n]$.
\end{claim}

This Claim  implies
\begin{equation*} 
(\ttt -  1)|\partial \HH| \geq \sum\nolimits_{Y\in {[n]\choose r-1}}   
\alpha(Y)=  |\HH|.
  \end{equation*}
This, together with~\eqref{eq:41} contradicts~\eqref{eq:12} and thus completes the proof of Theorem~\ref{th:main} in this case.

We prove Claim~\ref{distinct} in two steps in a stronger form which will be  useful in Section~\ref{stabi} to handle the stability of the extremal systems (i.e., Theorem~\ref{th:stab}).
For every $Y\subset [n]$, let $\UU(Y)$ be the set of vertices $v\in [n]\setminus Y$ such that there is an edge $E\in \HH$ containing $Y$ with $c(E)=v$.

\begin{claim}\label{distinct51}
Suppose $Y\subset [n]$, $a\leq |Y|\leq r-1$, $v,v'\in \UU(Y)$, $v\neq v'$ and edges $E,E'\in \HH$ are such that $v=c(E)$, $v'=c(E')$ and
$Y\subseteq E\cap E'$.
If
$E\in \Hh_i$ and $E'\in \Hh_{i'}$, then $i\neq i'$.
\end{claim}

\begin{proof} Suppose  $v\neq v'$, but $i=i'$.
We may assume that the partition of $[n]$ corresponding to $\Hh_i$ is $(X_1,\ldots,X_r)$ and
$v,v'\in X_r$.
Let $Z:=\{ j\in [r]: X_j\cap E\cap E'\neq \emptyset\}$.
By symmetry, we may also assume that $E\cap E' \subset X_1\cup\ldots\cup X_{|Z|}$.
So $Z\in \JJ_i$ and we know that $a\leq |Z|\leq r-1$.
The family $\JJ_i$ contains $ \mathcal{J}^{(r-1)}$, namely $[r]\setminus \{j\}\in \JJ_i$ for each $1\leq j\leq r-1$.
Since $\JJ_i$ is intersection closed it must contain every subset of $Z$, e.g., $[a]\in \JJ_i$,
 and it also contains every subset containg the element $r$, e.g., $[r]\setminus [a]\in \JJ_i$.
This contradicts Lemma~\ref{noa+b}.
\end{proof}

\begin{claim}\label{distinct52} Suppose $Y\subset [n]$, $a\leq |Y|\leq r-1$. Then $|\UU(Y)|\leq \ttt -1$.
\end{claim}

\begin{proof} Suppose to the contrary that there are $\ttt$ distinct $v_1,\ldots,v_\ttt \in [n]$ and
distinct  $E_1,\ldots,E_\ttt \in \HH$ such that $Y\subseteq E_1\cap\ldots\cap E_\ttt $ and $v_i=c(E_i)$ for $i=1,\ldots,\ttt$.
Let $E_1\in \mathcal{H}_{i_1}, \ldots, E_\ttt \in \mathcal{H}_{i_\ttt }$. 
By Claim~\ref{distinct51},  $i_1,\ldots,i_\ttt $ are all distinct. By relabelling we may suppose that $E_i\in \mathcal{H}_i$.

We will find  $\ttt +1$ disjoint sets $A_0, B_1,B_2,\ldots,B_\ttt $ contradicting Lemma~\ref{no+b} using induction as follows.
Fix a subset $A_0$ of $Y$ with $|A_0|=a$ and let $D_0:= A_0 \cup \{ c(E_1), \dots, c(E_\ttt )\}$. We have   $|D_0|=a+\ttt $.
We define the sets $E_i', D_i, B_i$ step by step as follows. We will have $D_i:=D_0\cup\bigcup_{j \leq i} E_j'$ and $|D_i|=a+\ttt +i(r-1-a)$.
For $i=1, 2, \dots, \ttt $ consider the family $\mathcal{H}_i$ and its member $E_i$ in it.
By the intersection structure of $\mathcal{H}_i$,
  the set $A_0\cup \{c(E_i)\}$ is an $(a+1,\ttt hr)$-kernel in $\mathcal{H}_i$.
One of the $\ttt hr$ petals of the sunflower in $\mathcal{H}_i$ with kernel $A_0\cup \{c(E_i)\}$
   should be disjoint from $D_{i-1}$; let $E_i'$ be the corresponding set in $\mathcal{H}_i$. Since $c(E_i)\in E_i'$, and homogeneity gives $c(E_i')=c(E_i)$, the set
   $B_i:=E_i'\setminus A_0$ is a $(b,\ttt hr)$-kernel.
\end{proof}

\subsection{The case $(a,b)=(2,2)$}\label{ss42}
Lemmas~\ref{noa+b} and~\ref{inter} imply that for each $1\leq i\leq m$, either\newline
---
$\mathcal{J}^{(3)}$ is contained in $\mathcal{J}(\mathcal{H}_i)$, it has exactly three 3-subsets, so $[4]\setminus\{1\}$,
$[4]\setminus\{2\}$, and $[4]\setminus\{3\}$
 are in $\mathcal{J}$ and $\mathcal{J}$ also contains
  all subsets containing the element $4$ but it does not contain $[3]$,  or \newline
---
$\mathcal{J}(\mathcal{H}_i)$ is of $ \mathcal{J}^{(1)}$ type, it has a unique 3-subset, $\{2,3,4\}$,  and $\{ \{1\}, \{1,2\},
\{1,3\}, \{1,4\}\}\subset \mathcal{J}$.

\smallskip
Call $\mathcal{H}_i$ (and its edges) {\em type} $\alpha$ if $\mathcal{J}(\mathcal{H}_i)$ has three 3-subsets.
Each of these edges $E\in\mathcal{H}_i$ has an element $c(E)\in E$ such that
 each proper subset of $E$ containing $c(E)$ is a kernel of a $4\cdot \ttt \cdot h$-star in $\mathcal{H}_i$.
The union of the families $\mathcal{H}_i$ of type $\alpha$  is $\HH_\alpha$.
Call $\mathcal{H}_i$ (and its edges) {\em type} $\beta$ if $\mathcal{J}(\mathcal{H}_i)$ has a unique 3-subset.
Each of these edges $E\in\mathcal{H}_i$ has an element $b(E)\in E$ such that
 each set $K\subset E$ of the form $\{b(E), x\}$ ($x\in E\setminus \{ b(E)\}$) and the set
 $E\setminus \{ b(E) \}$ is a kernel of a $4\cdot \ttt \cdot h$-star in $\mathcal{H}_i$.
The union of the families $\mathcal{H}_i$ of type $\beta$
 is $\HH_\beta$.

Define the function $\alpha$ on ${[n] \choose 3}$ as follows:
Given a $3$-set $Y$, let $\alpha(Y)$ be the number of edges $E\in \HH_\alpha$ with $Y=E\setminus \{ c(E)\}$.
Define the function $\beta$  on  ${[n] \choose 3}$ as follows:
Given a $3$-set $Y$, let $\beta(Y)$ be the number of edges $E\in \HH_\beta$ with $b(E)\in Y\subset E$.

\begin{claim}\label{distinct22}
$\alpha(Y)+\frac{1}{3} \beta(Y)\leq \ttt -1$, for all $3$-subsets $Y$ of $[n]$.
\end{claim}

\begin{proof} For brevity, let $\alpha=\alpha(Y)$ and $\beta=\beta(Y)$.
Let $E_1\in \mathcal{H}_{i_1}, \ldots, E_\alpha\in \mathcal{H}_{i_\alpha}$ be the $\alpha$ distinct edges $E$ with
  $E\in \HH_\alpha$ and $Y=E\setminus \{ c(E)\}$ and
let $E_{\alpha+1}\in \mathcal{H}_{i_{\alpha+1}}, \ldots, E_{\alpha+\beta}\in \mathcal{H}_{i_{\alpha+\beta}}$ be the $\beta$ distinct edges $E$ with $E\in \HH_\beta$ and $b(E)\in Y \subset E$.
If $i_j=i_{j'}$ for some $j\neq j'$,
 then the intersection structure of $\mathcal{H}_{i_j}$ would contain the $3$-set $Y$, a contradiction.
Thus $i_1, i_2,\ldots$ are all distinct. By relabelling we may suppose that $E_i\in \mathcal{H}_i$.

Suppose  that $\alpha+\frac{1}{3} \beta> \ttt -1$, so  $\alpha + \lceil \beta/3 \rceil \geq \ttt $.
Since $|Y|=3$, one can find an element $y_0\in Y$ such that $b(E_j)=y_0$ at least $\lceil \beta/3 \rceil$ times.
So we may suppose that there are $\ttt $ distinct $E_i\in \mathcal{H}_{i}$
  such that the elements $c(E_1), \dots, c(E_\alpha)$ and $d(E_j):= E_j\setminus Y$ for $\alpha < j\leq \ttt $ are all distinct
  and  $E_i=Y\cup \{ c(E_i)\}$, $b(E_j)=y_0$.

Let $A_0:= Y\setminus \{ y_0\}$. Then $|A_0|=2$ and we can find $\ttt +1$ disjoint sets $A_0, B_1,B_2,\ldots,B_\ttt $ contradicting Lemma~\ref{no+b} using induction in the same way as we did in the proof of Claim~\ref{distinct52}.
\end{proof}

Claim~\ref{distinct22} implies
\begin{equation*} 
(\ttt -  1)|\partial \HH| \geq \sum\nolimits_{Y\in {[n]\choose 3}} 
\left( \alpha(Y)+\frac{1}{3}\beta(Y)\right)=  |\HH_\alpha|+ |\HH_\beta|.
  \end{equation*}
This, together with~\eqref{eq:41}  contradicts~\eqref{eq:12} and thus completes the proof of Theorem~\ref{th:main} in this case.

\subsection{The case $(a,b)=(r-1,1)$}\label{ss43}

Call $\mathcal{H}_i$ (as above) of {\em type} $\alpha$ if $\mathcal{J}(\mathcal{H}_i)$ has $r-1$ $(r-1)$-subsets.
Each of these edges $E\in\mathcal{H}_i$ has an element $c(E)\in E$ such that
 each proper subset of $E$ containing $c(E)$ is a kernel of a $\ttt hr$-star in $\mathcal{H}_i$.
The union of these families  $\mathcal{H}_i$ is $\HH_\alpha$.
Call $\mathcal{H}_j$ (and its edges) {\em type} $\beta$ if $\mathcal{J}(\mathcal{H}_j)$ has no $(r-1)$-subset.
Note that each element $y$ of an edge $E\in\mathcal{H}_j$
 is a kernel of a $\ttt hr$-star in $\mathcal{H}_j$.
The union of these $\mathcal{H}_j$ families is $\HH_\beta$.

As in the previous subsections,  for each $Y\in {[n] \choose r-1}$ let
 $\alpha(Y)$ be the number of edges $E\in \HH_\alpha$ with $Y=E\setminus \{ c(E)\}$.
The definition of $\beta$ on  ${[n] \choose r-1}$ is even simpler:
  $\beta(Y)$ is the number of edges $E\in \HH_\beta$ with $Y\subset E$.
If $\alpha(Y)+\beta(Y)> \ttt -1$, then taking $A_0:=Y$ and $B_i:=E_i\setminus Y$, 
 each $B_i$ is a kernel of a large star,
contradicting Lemma~\ref{no+b}. Therefore, $\alpha(Y)+\beta(Y)\leq \ttt -1$ for each $Y\in {[n] \choose r-1}$, and
\begin{equation*} 
(\ttt -  1)|\partial \HH| \geq \sum_{Y\in {[n]\choose r-1}} \left( \alpha(Y)+\beta(Y)\right)=  |\HH_\alpha|+ r|\HH_\beta|\geq |\HH|,
  \end{equation*}
contradicting~\eqref{eq:12}.

\section{Hypergraphs without a bush $\BB_{\ttt, 1  }(1,r-1)$}\label{s5}

We will prove this part for  $C=C(r,\ttt ) :=  1/c(r,q)$, where
 $q:= 8r\ttt^2$
 and $c$ is from Lemma~\ref{Zoll}.
Suppose  $\mathcal{H}$ is a counter-example with the fewest edges. So
 $\Hh\subset {[n]\choose r}$, it is bush-free and $|\Hh|-|\partial \Hh|$ satisfies the lower bound~\eqref{eq:12}.
In particular,  $|\mathcal{H}|> C n^{r-2}$.

Call an $r$-graph $\GG$ {\em $\ttt $-normal} if it has no $(r-1)$-tuples of vertices whose  codegree is positive but less than $\ttt $.
 If 
 $\mathcal{H}$ is not $\ttt $-normal, then choose an $(r-1)$-tuple $Y$ of vertices whose  codegree is positive but less than $\ttt $ and
 let $\mathcal{H}'$ be obtained from $\mathcal{H}$ by deleting the
  edges containing  $Y$. Then $|\mathcal{H}'|-(\ttt -  1) |\partial \mathcal{H}'|\geq |\mathcal{H}|-(\ttt -  1) |\partial \mathcal{H}|>C n^{r-2}$, so $\mathcal{H}'$ satisfies~\eqref{eq:12} and is 
  $\BB_{\ttt, 1  }(1,r-1)$-free. This contradicts the minimality of $|\mathcal{H}|$. 
From now on, we suppose that $\mathcal{H}$ is $\ttt $-normal.

For every edge $Y\in\mathcal{H}$ and any $u\in Y$, let $Q(Y,u)=\{z\in V(\mathcal{H})\setminus 
 Y: Y\setminus \{u\}\cup \{z\}\in \mathcal{H}\}$.
Since $\mathcal{H}$ is $\ttt $-normal, $|Q(Y,u)|\geq \ttt -1$
 for every edge $Y\in \mathcal{H}$ and every  $u\in Y$.
For each  $u\in Y\in \Hh$, fix a subset
$Q'(Y,u)$ of $Q(Y,u)$ with $|Q'(Y,u)|=\min\{s,|Q(Y,u)|\}.$
Call a subfamily $\PP'\subset \Hh$ with $u\in \bigcap \PP'$ {\em separable} if $(\bigcup_{P\in \PP} Q'(P,u))\cap (\bigcup \PP')=\emptyset$.

\begin{claim}\label{separable}
Suppose that $u$ is a kernel of a star $\PP$ in $\mathcal H$, i.e.,
$\PP\subset \Hh$ such that
 $P_1\cap P_2 = u$ for all $P_1,P_2\in \PP$ whenever $P_1\neq P_2$.
Then there exists a separable $\PP'\subset \PP$ 
with $|\PP'|\geq |\PP|/(2s+1)$.
\end{claim}
\begin{proof} Let $G$ be the auxiliary directed graph with vertex set $\PP$ where the pair $\{ P_1, P_2\}\subset \PP$ is an arc if  
 $P_2\cap Q'(P_1,u)\neq \emptyset$. Let $G'$ be the underlying undirected graph of $G$.

Since $Q'(P,u)$ can meet at most $\ttt$ members of $\PP$, the
outdegree of each vertex  in $G$ is at most $\ttt$. So $G'$ is $2\ttt$-degenerate and hence $(2\ttt+1)$-colorable.
In particular,  $G'$ has an independent set of size at least $|\PP|/(2s+1)$.
An independent set in $G'$ corresponds to a separable subfamily of $\PP$.
\end{proof}

\begin{claim}\label{qstar}
Suppose that $u$ is a center of a separable star $\PP'$ in $\mathcal H$.
If $|\PP'|\geq r+2s-2$ then
 there exists a unique $(\ttt-1)$-element set $T(u)$ such that
 $Q(P, u)=T(u)$ for all $P\in \PP'$.\newline
 Moreover, $Q(Y,u)=T(u)$ for each $u\in Y\in \Hh$ with $Y\cap T(u)=\emptyset$. I.e.,
\begin{equation}\label{eq:51}
Y\setminus \{ u\} \cup \{ z\}\in \Hh \text{ for all }z\in T(u)\cup \{ u\}.
  \end{equation}
Let $T^+(u):= T(u)\cup \{ u\}$.
Then $T^+(z)$ is defined for each $z\in T^+(u)$ and coincides with $T^+(u)$.
\end{claim}
\begin{proof}
If there are $\ttt$ members of $\PP'$, say $P_1, \dots, P_\ttt$ such that $|\bigcup_{1\leq i\leq \ttt} Q'(P_i, u)| \geq \ttt$, then the family $\{Q'(P_1, u)$,...,$Q'(P_\ttt,u)\}$ satisfies Hall's condition. So, there exists a set of distinct representatives 
$\{z_1, \dots, z_\ttt\}$ such that $z_i\in Q'(P_i, u)$ for $1\leq i\leq \ttt$. Then the sets
$P_1, \dots, P_\ttt$ together with the hyperedges $P_1\setminus \{ u\}\cup \{ z_1\}, \dots,
  P_\ttt\setminus \{ u\}\cup \{ z_\ttt\}$
form a bush $\BB_{\ttt, 1  }(1,r-1)$ with central element $u$, a contradiction.

Hence $|\bigcup_{1\leq i\leq \ttt} Q'(P_i, u)| \leq \ttt-1$
 for any $\ttt$ distinct $P_1, \dots, P_\ttt\in \PP'$.
This implies $|Q'(P_i,u)|=\ttt -1$ for all $i$, so $|Q(P_i,u)|=\ttt -1$.
It also follows that $Q(P_i, u)=Q(P_j, u)$ for $1\leq i\leq j\leq \ttt$.
This holds for any pair from $\PP'$, so we get $Q(P, u)=Q(P_1, u)$ for all $P\in \PP'$.
Define $T(u):= Q(P_1,u)$.

Consider  $Y\in \Hh$ with $u\in Y$ and $Y\cap T(u)=\emptyset$. The set 
$(Y\setminus \{ u\})\cup Q'(Y,u)$ can meet at most $(r-1+\ttt)$ members of $\PP'$.
Since $|\PP'|\geq (r-1)+(2s-1)$  we can still find $P_1, \dots, P_{\ttt-1}\in \PP'$ such that
  $P_1, \dots, P_{\ttt-1}$ and $Y$ form a separable star.
This yields  $Q(Y, u)=T(u)$ and we are done.

The uniqueness of $T(u)$ also follows from the fact that having a $T_2(u)$ with similar properties one can find a $P\in \PP'$ avoiding it, so $T_2(u)=Q(P,u) =T(u)$.

To prove the last statement, choose $z\in T(u)$.  Equation~\eqref{eq:51} implies that the family
$\{ P\setminus \{u\} \cup \{ z\} :P\in \PP'\}$  is a separable star (we have $T^+(u)\setminus \{ z\} \subset Q(P\setminus \{u\} \cup \{ z\}, z)$).
So the first part of Claim~\ref{qstar} implies that $Q(P\setminus \{u\} \cup \{ z\}, z)= T^+(u)\setminus \{ z\} = T(z)$.
\end{proof}

Do the following procedure. Apply Lemma~\ref{Zoll} for $\Hh$ to get $\mathcal{H}_1=\mathcal{H}^*$
 with the corresponding intersection structure $\mathcal{J}_1\subset 2^{[r]}$.
For $i=1,2,\ldots$,  if
$|\mathcal{H}\setminus \bigcup_{j=1}^i \mathcal{H}_j|\leq C\cdot n^{r-2}$, then stop, let $m:=i$
and $\HH:=\bigcup_{j=1}^i \mathcal{H}_j$ and $\mathcal{H}_0=\mathcal{H}\setminus \HH$; otherwise,
let $\mathcal{H}_{i+1}:=(\mathcal{H}\setminus \bigcup_{j=1}^i \mathcal{H}_j)^*$.
We have $|\mathcal{H}_1|>n^{r-2}$ because $|\Hh|> Cn^{r-2}$ and by the choice of $C$. Similarly,
$|\Hh_i|> n^{r-2}$ for each $1\leq i\leq m$.
By Lemmas~\ref{noa+b} and~\ref{inter},  $\mathcal{J}_i$ contains a family isomorphic to $ \mathcal{J}^{(r-1)}$, or
 to $ \mathcal{J}^{(0)}$. In both cases, $\mathcal{J}_i$ contains a singleton, so $\Hh_i$ contains
 $q$-stars with singleton kernels. 

{\bf Case 1.} There exists a $\mathcal{J}_i$ containing a family isomorphic to $ \mathcal{J}^{(r-1)}$.
\newline
We may assume that $\mathcal{J}_i$ contains all proper subsets of $[r]$ containing $1$ and $X_1, \dots, X_r$ are the parts of $\Hh_i$.
There is an element $y_1\in X_1$ such that $\{y_1\}$ is a kernel of a $q$-star in $\Hh_i$.
Claims~\ref{separable} and~\ref{qstar} imply the existence of $T(y_1)$ since $\frac{q}{2s+1}\geq r+2s-2$.
We can choose a $Y_1=\{y_1,\ldots,y_r\}\in \mathcal{H}_i$ where $y_j\in X_j$ for $j\in [r]$ with
 $Y\cap T(y_1)=\emptyset$.

Since $\{1,2\}\in \mathcal{J}_i$, $\{y_1, y_2\}$ is the kernel of a $2\ttt$-star $Y_1, \dots, Y_{2\ttt}\in \Hh_i$ (i.e, the sets $Y_j\setminus \{y_1, y_2\}$ are pairwise disjoint for $j\in [2\ttt]$). 
At most $\ttt -1$ of them intersect $T(y_1)$. So we may assume, e.g.,
$Y_1,\ldots,Y_\ttt $ are disjoint from $T(y_1)$.
Since $Y_j\setminus \{ y_2\}$ is a kernel of a $q$-star for each $j\in [\ttt]$,
 one can find distinct elements $x_1, \dots, x_\ttt$ from $X_2$ such that none of them lies in $T(y_1)$ and $Y_j\setminus \{ y_2\} \cup \{ x_j\}\in \Hh_i$.
Let $T^+(y_1):= T(y_1)\cup \{ y_1\}=\{ z_1, \dots, z_\ttt \}$ and
apply~\eqref{eq:51} for $Y_j\setminus \{ y_2\} \cup \{ x_j\}$.
We obtain that $Y_j\setminus \{y_1, y_2\} \cup \{ z_j, x_j\}\in \Hh$.

Apply~\eqref{eq:51} for $Y_j$ (again with $T^+(y_1)$).
We get that $Y_j\setminus \{ y_1\}\cup \{z_j \} \in \Hh$.
These edges, together with the edges $Y_j\setminus \{y_1, y_2\} \cup \{ z_j, x_j\}$ form a bush
 $\BB_{\ttt,1  }(1,r-1)$ with
 the central vertex $y_2$.
This contradiction leads to the last case.

{\bf Case 2.}
$\mathcal{J}_i$ contains a family isomorphic to $ \mathcal{J}^{(0)}$ for each $1\leq i\leq m$.

By the structure of $ \mathcal{J}^{(0)}$, each $v\in V(\mathcal{H}_i)$ is a kernel 
of a $q$-star in $\Hh_i$. So by Claims~\ref{separable} and~\ref{qstar} for each $u\in \bigcup \HH$,
  $T^+(u)$ is well defined  and $|T^+(u)|=s$.

Since $ \mathcal{J}^{(0)}$ does not contain sets of size $r-1$, each $(r-1)$-element set $Y\in \partial \mathcal{H}_i$
is only in one set in $\mathcal{H}_i$, thus $|\partial \mathcal{H}_i| =r|\mathcal{H}_i|$. 
As $|\mathcal{H}_0|\leq C n^{r-2}$ equation~\eqref{eq:12} implies $|\HH|>(\ttt -  1)  |\partial \mathcal{H}| \geq (\ttt -  1)  |\partial \HH|$.
We obtain $\sum_i |\partial \mathcal{H}_i|= r\sum_i |\mathcal{H}_i| = r|\HH|>r(\ttt -  1) |\partial \HH|$.
Hence some $(r-1)$-tuple $S$ belongs to at least $r(\ttt -  1)+1$ shadow families $\partial\mathcal{H}_i$.
Say, $S\cup \{z_i\}\in \mathcal{H}_i$ for $i\in [r(\ttt -  1)+1]$.
Let $Z=\{z_1,\ldots,z_{r(\ttt -  1)+1}\}$.

For every $y\in S$,   $|T(y)|= \ttt -1$;   
 thus there is $z_j\in Z\setminus \bigcup_{y\in S} T^+(y)$, say $z_1$.
Let $Y_1:= S\cup \{ z_1\}$.
We got that $T(z_1)$ is disjoint from $Y_1$ while $z_1\in Y_1\in \Hh_1$. So Claim~\ref{qstar} implies that $Q(Y_1,z_1)=T(z_1)$. However $Q(Y_1,z_1)$ contains $Z\setminus \{ z_1\}$.

This final contradiction implies that the minimal counterexample $\Hh$ does not exist, completing the proof of Theorem~\ref{th:main}.

\section{Stability: Proof of Theorem~\ref{th:stab}}\label{stabi}

Let $\mathcal{H}$ be an $n$-vertex $r$-uniform family not containing
 a bush  $\BB_{\ttt  ,h}(a,b)$ with
 \begin{equation}\label{eq:60'}
 |\mathcal{H}|> (\ttt -  1){n\choose {r-1}}-C_0 n^{r-2}.
   \end{equation}
Define $C$,  $m$, $\Hh_0,\ldots,\Hh_m$ and $\HH$ as in Subsection~\ref{41}. By~\eqref{eq:60'} and the definition of $\Hh_0$,
\begin{equation}\label{eq:61}
|\HH|    \geq (\ttt -  1) {n\choose {r-1}}-(C+C_0)n^{r-2}.
   \end{equation}

As in Subsection~4.2, for each $1\leq i\leq m$, the intersection structure $\mathcal{J}(\mathcal{H}_i)$ contains $ \mathcal{J}^{(r-1)}$. So, again
 for each hyperedge $E\in \mathcal{H}_i\subset \HH $ ($1\leq i\leq m$) there is an element $c(E)\in E$ such that each proper subset of $E$ containing $c(E)$ is a kernel of an $\ttt hr$-star in  $\mathcal{H}_i$.

For $v\in [n]$, let $\Hh(v)=\{E\in \HH: v=c(E)\}$ and $\GG(v)=\{E\setminus \{v\}: E\in \Hh(v)\}$.  Let $\GG=\bigcup_{v\in [n]}\GG(v)$.
By Claim~\ref{distinct51},
$|\GG(v)|=|\HH(v)|$ for each $v\in [n]$; in particular, $\sum_{v} |\GG(v)|=|\HH|$.
Furthermore, Claim~\ref{distinct52} implies that 
\begin{equation}\label{eq:62}
\mbox{\em each  $(r-2)$-subset $Y$ of $[n]$ is in the shadow of at most $\ttt -1$ families  $\GG(v)$.
}
   \end{equation}
Recall the  Lov\' asz form of the Kruskal-Katona Theorem:
\begin{theorem}\label{th:LKK}
If $x$ is a positive real, $1\leq k<n$, $\FF\subseteq {[n]\choose k}$ and $|\FF|= {x\choose k}$, then
$|\partial \FF|\geq {x\choose k-1}$.
\end{theorem}
Here ${x\choose k}$ is a non-negative real convex function defined as a degree $k$ polynomial
 $x(x-1)\dots (x-k+1)/k!$ for $x\geq k-1$ and $0$ otherwise.

For every $j\in [n]$, there a real $x_j$ such that $|\GG(j)|={x_j\choose r-1}$.
Inequality~\eqref{eq:61} gives
\begin{equation}\label{eq:63x}
   \sum_i {x_i\choose r-1} \geq (\ttt -  1) {n\choose {r-1}}-(C+C_0)n^{r-2},
   \end{equation}
and~\eqref{eq:62} and Lov\'asz theorem  give
\begin{equation}\label{eq:64x}
   (\ttt-1){n\choose r-2} \geq  \sum_i |\partial\GG(j)| \geq \sum_i {x_i\choose r-2}.
   \end{equation}
 Let $b=3 (r-1)!(C+C_0)$.

\begin{lemma}\label{d53} Suppose that $n\geq r^2$ 
and  $x_1\geq x_2\geq \ldots \geq x_n$. Then
inequalities~\eqref{eq:63x} and~\eqref{eq:64x} concerning $n,r, x_1, \dots, x_n$ imply
 $x_1, \dots, x_{\ttt-1} > n-b$.
\end{lemma}

\begin{proof} Pure college algebra. For brevity let $y:=x_{\ttt}$.
The case $y\leq r-2$ is obvious (in this case the left hand side of~\eqref{eq:63x} has at most $\ttt-1$
nonzero terms), so from now on we suppose that $y> r-2$.
Multiply~\eqref{eq:64x} by $(y-r+2)/(r-1)$ and add it to~\eqref{eq:63x}.
Using ${x\choose r-1} = (x-r+2)/(r-1) \times {x\choose r-2}$ (for all reals $x> r-2\geq 1$)
after rearrangements we get
\begin{eqnarray*}
(C+C_0)n^{r-2}&\geq& \sum_{i\geq \ttt} \frac{y-x_i}{r-1}{x_i\choose r-2}\\
  &&+ \sum _{1\leq i \leq \ttt-1}\frac{n-x_i}{r-1}{n\choose r-2}\\
  &&+ \sum _{1\leq i \leq \ttt-1}\frac{x_i-y}{r-1}\left( {n\choose r-2}-{x_i\choose r-2} \right).
\end{eqnarray*}
The first and the third rows are non-negative. We obtain
$$ (r-1)(C+C_0)n^{r-2} / {n\choose r-2} \geq \sum_{1\leq i\leq \ttt-1} (n-x_i).
  $$
So the left hand side is at most $b$ for $n\geq r^2$.
\end{proof}

Now we are ready to finish the proof of the theorem. By Lemma~\ref{d53}, for every $1\leq i\leq \ttt -1$,
$$|\HH(i)|={x_i\choose r-1}\geq {n-b\choose r-1} 
\geq {n\choose r-1}-b{n\choose r-2}\geq {n\choose r-1}-\frac{b}{(r-1)!}n^{r-2}.
$$
This completes the proof of Theorem~\ref{th:stab} for $C_1=\frac{b}{(r-1)!}=3(C+C_0)$.

\paragraph{Acknowledgment.} We thank a referee for helpful comments.

\end{document}